\documentclass[12pt]{amsart}
\usepackage{amsfonts, amsmath, amsthm}

\usepackage{graphicx}
\usepackage{epstopdf}
\usepackage{psfrag}

\newtheorem{pro}{Proposition}[section]
\newtheorem{thm}[pro]{Theorem}
\newtheorem{lem}[pro]{Lemma}
\newtheorem{clm}[pro]{Claim}

\newtheorem{cor}[pro]{Corollary}

\theoremstyle{definition}
\newtheorem{dfn}[pro]{Definition}

\theoremstyle{remark}

\newcommand{\VV}{\mathcal V}
\newcommand{\WW}{\mathcal W}

\newcommand{\bdy}{\partial}

\title{Barriers to Topologically Minimal Surfaces}
\date{\today}
\address{Pitzer College}
\email{bachman@pitzer.edu}
\author{David Bachman}
\begin{document}

\begin{abstract}
In earlier work we introduced topologically minimal surfaces as the analogue of geometrically minimal surfaces. Here we strengthen the analogy by showing that complicated amalgamations act as barriers to low genus, topologically minimal surfaces.
\end{abstract}

\maketitle


\section{Introduction}

Suppose $X$ and $Y$ are compact, orientable, irreducible 3-manifolds with incompressible, homeomorphic boundary. Let $\phi:\bdy X \to \bdy Y$ denote any homeomorphism, and $\psi:\bdy X \to \bdy X$ some pseudo-Anosov. Let $M_n$ denote the manifolds obtained from $X$ and $Y$ by gluing their boundaries via the maps $\phi \psi^n$. Let $F$ denote the image of $\bdy X$ in $M_n$. For increasing values of $n$ we see longer and longer regions homeomorphic to $F \times I$ in $M_n$. Any minimal surface embedded in $M_n$ that crosses through such a region will have large area, and thus large genus by the Gauss-Bonnet Theorem. Hence, if $M_n$ contains a low genus minimal surface, it must be contained in $X$ or $Y$. The mantra is this: complicated amalgamations act as barriers to low genus minimal surfaces. 

In \cite{TopIndexI} we defined a surface with empty or non-contractible disk complex to be {\it topologically minimal}. Such surfaces generalize more familiar classes, such as {\it incompressible}, {\it strongly irreducible}, and {\it critical}. The main result of the present paper is given in Theorem \ref{t:DistanceBoundTheorem}, which says that complicated amalgamations act as barriers to low genus, topologically minimal surfaces. This gives further evidence that topologically minimal surfaces are the appropriate topological analogue to geometrically minimal surfaces. 

Fix incompressible, $\bdy$-incompressible surfaces $Q_X \subset X$ and $Q_Y \subset Y$ that have maximal Euler characteristic. Here we measure the complexity of any gluing map $\phi: \bdy X \to \bdy Y$ by the distance, in the curve complex of $\bdy Y$, between $\phi(\bdy Q_X)$ and $\bdy Q_Y$. (When $\bdy Y \cong T^2$ then distance is measured in the Farey graph.) So, for example, the gluing maps used to construct the manifolds $M_n$ above increase in complexity as $n$ increases. 

A rough sketch of the proof of Theorem \ref{t:DistanceBoundTheorem} is as follows. In \cite{TopIndexI} we showed that any topologically minimal surface $H$ can be isotoped to meet the gluing surface $F$ so that $H_X=H \cap X$ and $H_Y=H \cap Y$ are topologically minimal in $X$ and $Y$. In Section \ref{s:IndexRelBdy} we show that  if $H$ is not isotopic into $Y$, then $H_X$ can be $\bdy$-compressed in $X$ to a surface $H'_X$ in a more restrictive class. Such surfaces are said to be topologically minimal {\it with respect to $\bdy X$}. A short argument then shows that the distance between the loops of $\bdy H_X$ and $\bdy H'_X$ is bounded by a function of the difference between the Euler characteristics of these surfaces. In Section \ref{s:DistanceBoundSection} we apply a technique of Tao Li \cite{li:08} to show that the distance between the loops of $\bdy H'_X$ and the loops of $\bdy Q_X$ is bounded by an explicit function of the Euler characteristics of $H'_X$ and $Q_X$. 

Similar statements can be made about the surface $H_Y$. That is, if $H$ is not isotopic into $X$ then there is a surface $H'_Y$, obtained from $H_Y$ by $\bdy$-compressing in $Y$, such that $\bdy H'_Y$ is a bounded distance away from $\bdy Q_Y$. Putting all of this information together then says that if $H$ is not isotopic into $X$ or $Y$, then the distance between $\bdy Q_X$ and $\bdy Q_Y$ is bounded by some constant $K$ times the genus of $H$, where $K$ is an explicit linear function depending only on the Euler characteristics of $Q_X$ and $Q_Y$. 

Theorem \ref{t:DistanceBoundTheorem} also holds in more generality. Another situation to which it applies is when $X$ is a connected 3-manifold with homeomorphic boundary components $F_-$ and $F_+$. Then the theorem asserts that if $F_-$ is glued to $F_+$ by a sufficiently complicated map, then any low genus, topologically minimal surface in the resulting 3-manifold is isotopic into $X$. 

The index 0, 1, and 2 cases of Theorem \ref{t:DistanceBoundTheorem} play a central role in the two sequels to this paper, \cite{StabilizationResults} and \cite{AmalgamationResults}. In these papers we answer a variety of questions dealing with stabilization and isotopy of Heegaard surfaces. The first contains a construction of a pair of inequivalent Heegaard splittings that require several stabilizations to become equivalent. In the second, we show that when low genus, unstabilized, $\bdy$-unstabilized Heegaard surfaces are amalgamated via a ``sufficiently complicated" gluing, the result is an unstabilized Heegaard surface. 

The author thanks Tao Li for helpful conversations regarding his paper, \cite{li:08}, on which Section \ref{s:DistanceBoundSection} is based. Comments from Saul Schleimer were also helpful.

\section{Topologically minimal surfaces.}

In this section we review the definition of the {\it topological index} of a surface, given in \cite{TopIndexI}. Let $H$ be a properly embedded, separating surface with no torus components in a compact, orientable 3-manifold $M$. Then the {\it disk complex},  $\Gamma(H)$, is defined as follows:
	\begin{enumerate}
		\item  Vertices of $\Gamma(H)$ are isotopy classes of compressions for $H$. 
		\item A set of $m+1$ vertices forms an $m$-simplex if there are representatives for each that are pairwise disjoint. 
	\end{enumerate}

\begin{dfn}
\label{d:Indexn}
The {\it homotopy index} of a complex $\Gamma$ is defined to be 0 if $\Gamma=\emptyset$, and the smallest $n$ such that $\pi_{n-1}(\Gamma)\ne 1$, otherwise. We say the {\it topological index} of a surface $H$ is the homotopy index of its disk complex, $\Gamma(H)$. If $H$ has a topological index then we say it is {\it topologically minimal}. 
\end{dfn}

In \cite{TopIndexI} we show that incompressible surfaces have topological index 0, strongly irreducible surfaces (see \cite{cg:87}) have topological index 1, and critical surfaces (see \cite{crit}) have topological index 2. In \cite{existence} we show that for each $n$ there is a manifold that contains a surface whose topological index is $n$.


\begin{thm}[\cite{TopIndexI}, Theorem 3.8.]
\label{t:OriginalIntersection}
Let $F$ be a properly embedded, incompressible surface in an irreducible 3-manifold $M$. Let $H$ be a properly embedded surface in $M$ with topological index $n$. Then $H$ may be isotoped so that
	\begin{enumerate}
		\item $H$ meets $F$ in $p$ saddles, for some $p \le n$, and 
		\item the topological index of $H \setminus N(F)$ in $M \setminus N(F)$, plus $p$, is at most $n$. 
	\end{enumerate}
\end{thm}



In addition to this result about topological index, we will need the following:

\begin{lem}
\label{l:EssentialBoundary}
Suppose $H$ is a topologically minimal surface which is properly embedded in a 3-manifold $M$ with incompressible boundary. Then each loop of $\bdy H$ either bounds a component of $H$ that is a boundary-parallel disk, or is essential on $\bdy M$.
\end{lem}

\begin{proof}
Begin by removing from $H$ all components that are boundary-parallel disks. If nothing remains, then the result follows. Otherwise, the resulting surface (which we continue to call $H$) is still topologically minimal, as it has the same disk complex. Now, let $\alpha$ denote a loop of $\bdy H$ that is innermost among all such loops that are inessential on $\bdy M$ . Then $\alpha$ bounds a compression $D$ for $H$ that is disjoint from all other compressions. Hence, every maximal dimensional simplex of $\Gamma(H)$ includes the vertex corresponding to $D$. We conclude $\Gamma(H)$ is contractible to $D$, and thus $H$ was not topologically minimal. 
\end{proof}

\section{Topological index relative to boundaries.}
\label{s:IndexRelBdy}

In this section we define the topological index of a surface $H$ in a 3-manifold $M$ {\it with respect to $\bdy M$.} We then show that we may always obtain such a surface from a topologically minimal surface by a sequence of $\bdy$-compressions. 

Let $H$ be a properly embedded, separating surface with no torus components in a compact, orientable 3-manifold $M$. Then the complex $\Gamma(H;\bdy M)$, is defined as follows:
	\begin{enumerate}
		\item  Vertices of $\Gamma(H;\bdy M)$ are isotopy classes of compressions and $\bdy$-compressions for $H$. 
		\item A set of $m+1$ vertices forms an $m$-simplex if there are representatives for each that are pairwise disjoint. 
	\end{enumerate}

\begin{dfn}
\label{d:Indexn}
We say the topological index of a surface $H$ {\it with respect to $\bdy M$} is the homotopy index of the complex $\Gamma(H;\bdy M)$. If $H$ has a topological index with respect to $\bdy M$ then we say it is {\it topologically minimal with respect to $\bdy M$}. 
\end{dfn}

In Corollary 3.9 of \cite{TopIndexI} we showed that a topologically minimal surface can always be isotoped to meet an incompressible surface in a collection of essential loops. The exact same argument, with the words ``compression or $\bdy$-compression" substituted for ``compression" and ``innermost loop/outermost arc" substituted for ``innermost loop," shows:

\begin{lem}
\label{l:EssentialIntersection}
Let $M$ be a compact, orientable, irreducible 3-manifold with incompressible boundary. Let $H$ and $Q$ be properly embedded surfaces in $M$, where $H$ is topologically minimal with respect to $\bdy M$ and $Q$ is both incompressible and $\bdy$-incompressible. Then $H$ may be isotoped so that it meets $Q$ in a (possibly empty) collection of loops and arcs that are essential on both.  
\end{lem}

\begin{dfn}
\label{d:H/D}
If $D$ is a compression or $\bdy$-compression for a surface $H$ then we construct the surface $H/D$ as follows. Let $M(H)$ denote the manifold obtained from $M$ by cutting open along $H$. Let $B$ denote a neighborhood of $D$ in $M(H)$. The surface $H/D$ is obtained from $H$ by removing $B \cap H$ and replacing it with the frontier of $B$ in $M(H)$. The surface $H/D$ is said to have been obtained from $H$ by {\it compressing} or {\it $\bdy$-compressing} along $D$. Similarly, suppose $\tau$ is some simplex of $\Gamma(H;\bdy M)$ and $\{D_i\}$ are pairwise disjoint representatives of the vertices of $\tau$. Then $H/\tau$ is defined to be the surface obtained from $H$ by simultaneously compressing/$\bdy$-compressing along each disk $D_i$. 
\end{dfn}

We leave the proof of the following lemma as an exercise for the reader.

\begin{lem}
\label{l:CompressionEffect}
Suppose $M$ is an irreducible 3-manifold with incompressible boundary. Let $D$ be a $\bdy$-compression for a properly embedded surface $H$. Then $\Gamma(H/D)$ is the subset of $\Gamma(H)$ spanned by those compressions that are disjoint from $D$. \qed
\end{lem}

\begin{thm}
\label{t:MainTheorem}
Suppose $M$ is an irreducible 3-manifold with incompressible boundary.  Let $H$ be a surface whose topological index is $n$. Then either $H$ has topological index at most $n$ with respect to $\bdy M$, or there is a simplex $\tau$ of $\Gamma(H;\bdy M) \setminus \Gamma(H)$, such that the topological index of $\Gamma(H/\tau)$ is at most $n-{\rm dim}(\tau)$. 
\end{thm}


\begin{proof}
If $\Gamma(H) = \emptyset$ then the result is immediate, as any surface obtained by $\bdy$-compressing an incompressible surface must also be incompressible. 

If $\Gamma(H) \ne \emptyset$ then, by assumption, there is a non-trivial map $\iota$ from an $(n-1)$-sphere $S$ into the $(n-1)$-skeleton of  $\Gamma(H)$. Assuming the theorem is false will allow us to inductively construct a map  $\Psi$ of a $n$-ball $B$ into $\Gamma(H)$ such that $\Psi(\bdy B)=\iota(S)$. The existence of such a map contradicts the non-trivialty of $\iota$. 


Note that $\Gamma(H) \subset \Gamma(H;\bdy M)$. If  $\pi_{n-1}(\Gamma(H;\bdy M)) \ne 1$ then the result is immediate. Otherwise,  $\iota$ can be extended to a map from  an $n$-ball $B$ into $\Gamma(H;\bdy M)$.

Let $\Sigma$ denote a triangulation of $B$ so that the map $\iota$ is simplicial. If $\tau$ is a simplex of $\Sigma$ then we let $\tau^{\bdy}$ denote the vertices of $\tau$ whose image under $\iota$ represent $\bdy$-compressions. If $\tau^{\bdy}=\emptyset$ then let $V_\tau=\Gamma(H)$.  Otherwise, let $V_\tau$ be the subspace of $\Gamma(H)$ spanned by the compressions that can be made disjoint from the disks represented by every vertex of $\iota(\tau^{\bdy})$. In other words, $V_\tau$ is the intersection of the link of $\tau^\bdy$ in $\Gamma(H;\bdy M)$ with $\Gamma(H)$. 

By Lemma \ref{l:CompressionEffect}, when $\tau^{\bdy} \ne \emptyset$ then $\Gamma(H/\tau^{\bdy})$ is precisely $V_\tau$. (More precisely, there is an embedding of $\Gamma(H/\tau^{\bdy})$ into $\Gamma(H)$ whose image is $V_\tau$.) By way of contradiction, we suppose the homotopy index of $\Gamma(H/\tau^{\bdy})$ is not less than or equal to  $n-{\rm dim}(\tau^{\bdy})$. Thus, $V_\tau \ne \emptyset$ and 
\begin{equation}
\label{e:contradictionRevisited}
\pi_i(V_\tau) =1 \mbox{ for all } i \le  n-1-{\rm dim}(\tau^{\bdy}). 
\end{equation}

\begin{clm}
\label{c:subsetRevisited}
Suppose $\tau$ is a cell of $\Sigma$ which lies on the boundary of a cell $\sigma$. Then $V_{\sigma} \subset V_{\tau}$. 
\end{clm}

\begin{proof}
Suppose $D \in V_{\sigma}$. Then $D$ can be made disjoint from disks represented by every vertex of $\iota(\sigma^{\bdy})$. Since $\tau$ lies on the boundary of $\sigma$, it follows that $\tau^{\bdy} \subset \sigma^{\bdy}$. Hence, $D$ can be made disjoint from the disks represented by every vertex of $\iota(\tau^{\bdy})$. It follows that $D \in V_\tau$. 
\end{proof}

Push the triangulation $\Sigma$ into the interior of $B$, so that $\rm{Nbhd}(\bdy B)$ is no longer triangulated (Figure \ref{f:SigmaDual}(b)). Then extend $\Sigma$ to a cell decomposition over all of $B$ by forming the product of each cell of $\Sigma \cap S$ with the interval $I$ (Figure \ref{f:SigmaDual}(c)). Denote this cell decomposition as $\Sigma'$. Note that $\iota$ extends naturally over $\Sigma'$, and the conclusion of Claim \ref{c:subsetRevisited} holds for cells of $\Sigma'$. Now let $\Sigma^*$ denote the dual cell decomposition of $\Sigma'$ (Figure \ref{f:SigmaDual}(d)). This is done in the usual way, so that there is a correspondence between the $d$-cells of $\Sigma ^*$ and the $(n-d)$-cells of $\Sigma'$. Note that, as in the figure, $\Sigma ^*$ is not a cell decomposition of all of $B$, but rather a slightly smaller $n$-ball, which we call $B'$. 

\begin{figure}
\psfrag{a}{(a)}
\psfrag{b}{(b)}
\psfrag{c}{(c)}
\psfrag{d}{(d)}
\begin{center}
\includegraphics[width=5 in]{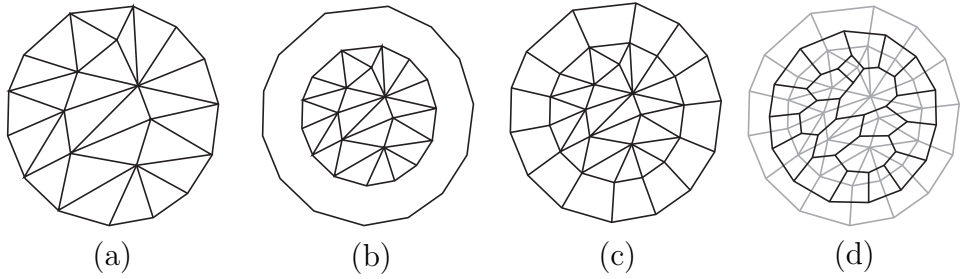}
\caption{(a) The triangulation $\Sigma$ of $B$. (b) Push $\Sigma$ into the interior of $B$. (c) Fill in $\rm{Nbhd}(\bdy B)$ with product cells to complete $\Sigma'$. (d) $\Sigma^*$ is the dual of $\Sigma'$.}
\label{f:SigmaDual}
\end{center}
\end{figure}

For each cell $\tau$ of $\Sigma'$, let $\tau^*$ denote its dual in $\Sigma^*$. Thus, it follows from Claim \ref{c:subsetRevisited} that if $\sigma^*$ is a cell of $\Sigma^*$ that is on the boundary of $\tau^*$, then $V_{\sigma} \subset V_\tau$. 

We now produce a contradiction by defining a continuous map $\Psi:B' \to \Gamma(H)$ such that $\Psi(\bdy B')=\iota(S)$. The map will be defined inductively on the $d$-skeleton of $\Sigma^*$ so that the image of every cell $\tau^*$ is contained in $V_\tau$. 

For each $0$-cell $\tau^* \in \Sigma^*$, choose a point in $V_\tau$ to be $\Psi(\tau^*)$. If $\tau^*$ is in the interior of $B'$ then this point may be chosen arbitrarily in $V_\tau$. If $\tau^* \in \bdy B'$ then $\tau$ is an $n$-cell of $\Sigma'$. This $n$-cell is $\sigma \times I$, for some $(n-1)$-cell $\sigma$ of $\Sigma \cap S$. But since $\iota(S) \subset \Gamma(H)$, it follows that $\tau^{\bdy}=\sigma ^{\bdy}=\emptyset$, and thus $V_\tau=\Gamma(H)$. We conclude $\iota(\tau) \subset V_\tau$, and thus we can choose $\tau^*$, the barycenter of $\iota(\tau)$, to be $\Psi(\tau^*)$. 

We now proceed to define the rest of the map $\Psi$ by induction. Let $\tau^*$ be a $d$-dimensional cell of $\Sigma^*$ and assume $\Psi$ has been defined on the $(d-1)$-skeleton of $\Sigma^*$. In particular, $\Psi$ has been defined on $\bdy \tau^*$. Suppose $\sigma^*$ is a cell on $\bdy \tau^*$.  By Claim \ref{c:subsetRevisited} $V_\sigma \subset V_\tau$. By assumption $\Psi|\sigma^*$ is defined and $\Psi(\sigma^*) \subset V_\sigma$. We conclude $\Psi(\sigma^*) \subset V_\tau$ for all $\sigma^* \subset \bdy \tau^*$, and thus
\begin{equation}
\label{e:boundary}
\Psi(\bdy \tau^*) \subset V_\tau.
\end{equation}

Note that  \[{\rm dim}(\tau) = n-{\rm dim}(\tau^*)=n-d.\] Since ${\rm dim}(\tau^{\bdy}) \le {\rm dim}(\tau)$, we have \[{\rm dim}(\tau^{\bdy}) \le n-d.\] Thus \[d \le n-{\rm dim}(\tau^{\bdy}),\] and finally \[d-1 \le n-1-{\rm dim}(\tau^{\bdy}).\]
It now follows from Equation \ref{e:contradictionRevisited} that $\pi_{(d-1)}(V_\tau)=1$. Since $d-1$ is the dimension of $\bdy \tau^*$, we can thus extend $\Psi$ to a map from $\tau^*$ into $V_\tau$. 

Finally, we claim that if $\tau^* \subset \bdy B'$ then this extension of $\Psi$ over $\tau^*$ can be done in such a way so that $\Psi(\tau^*) =\iota(\tau^*)$. This is because in this case each vertex of $\iota(\tau)$ is a compression, and hence $V_\tau=\Gamma(H)$. As $\iota(S) \subset \Gamma(H)=V_\tau$, we have $\iota(\tau^*) \subset V_\tau$. Thus we may choose $\Psi(\tau^*)$ to be $\iota(\tau^*)$.
\end{proof}

\begin{cor}
\label{c:FirstDistanceBound}
Suppose $M$ is a compact irreducible 3-manifold with incompressible boundary. Let $F$ denote a component of $\bdy M$. Let $H$ be a surface whose topological index is $n$. Then either $H$ is isotopic into a neighborhood of $\bdy M$ or there is a surface $H'$, obtained from $H$ by some sequence of (possibly simultaneous) $\bdy$-compressions, such that 
	\begin{enumerate}
		\item $H'$ has topological index at most $n$ with respect to $\bdy M$, and 
		\item The distance between $H \cap F$ and $H' \cap F$ is at most $3\chi(H')-3\chi(H)$.
	\end{enumerate}
\end{cor}

\begin{proof}
We first employ Theorem \ref{t:MainTheorem} to obtain a sequence of surfaces, $\{H_i\}_{0=1}^m$ in $M$, such that 
	\begin{enumerate}
		\item $H_0=H$
		\item $H_{i+1}=H_i/\tau_i$, for some simplex $\tau_i \subset \Gamma(H_i; \bdy M) \setminus \Gamma(H_i)$.
		\item For each $i$ the topological index of $H_{i+1}$ in $M$ is $n_{i+1} \le n_i-\rm{dim}(\tau_i)$, where $n_0=n$.
		\item $H_m$ has topological index at most $n_m$ with respect to $\bdy M$. 
	\end{enumerate}

It follows from Lemma \ref{l:EssentialBoundary} that $\bdy H_i$ contains at least one component that is essential on $\bdy M$ for each $i<m$. The boundary of $H_m$ is essential if and only if $H_m$ is not a collection of $\bdy$-parallel disks. But in this case, the surface $H$ was isotopic into a neighborhood of $\bdy M$. Hence, it suffices to show that for all $i$, the distance between the loops of $H_i \cap F$ and $H_{i+1} \cap F$ is bounded by $3\chi(H_{i+1})-3\chi(H_i)$.

Let $\VV$ and $\WW$ denote the sides of $H_i$ in $M$. If the dimension of $\tau_i$ is $k-1$, then $H_{i+1}$ is obtained from $H_i$ by $k$ simultaneous $\bdy$-compressions. Hence, the difference between the Euler characteristics of $H_i$ and $H_{i+1}$ is precisely $k$. 

If $k=1$ then the loops of $\bdy H_i$ can be made disjoint from the loops of $\bdy H_{i+1}$. It follows that 
\[d(H_i \cap F, H_{i+1} \cap F) \le 1=k < 3k,\]
and thus the result follows. Henceforth, we will assume $k \ge 2$. 

Let $\{V_1,...,V_p\}$ denote the $\bdy$-compressions represented by vertices of $\tau_i$ that lie in $\VV$, and $\{W_1,...,W_q\}$ the $\bdy$-compressions represented by vertices of $\tau_i$ that lie in $\WW$. We will assume that $\VV$ and $\WW$ were labelled so that $p \le k/2$. The loops of $H_{i+1}\cap F$ are obtained from the loops of $H_i \cap F$ by band sums along the arcs of $V_i \cap F$ and $W_i \cap F$. By pushing the loops of $H_i \cap F$ slightly into $\VV$, we see that they meet the loops of $H_{i+1} \cap F$ at most $4p$ times. That is, \[|H_i \cap H_{i+1}\cap F| \le 4p \le 2k.\]

When $F$ is not a torus we measure the distance between curves $\alpha$ and $\beta$ in $F$ in its curve complex. By \cite{hempel:01}, this distance is bounded above by $2+2\log _2 |\alpha \cap \beta|$. Hence, the distance between $H_i \cap F$ and $H_{i+1} \cap F$ is at most $2+2\log_2 2k$. But for any positive integer $k$, $\log _2 2k \le k$. Hence, we have shown $d(H_i \cap F , H_{i+1} \cap F) \le 2+2k$. Since $k \ge 2$ it follows that  $2+2k \le 3k$, and thus the result follows.

When $F \cong T^2$, the distance between curves $\alpha$ and $\beta$ is measured in the Farey graph. In this case their distance is bounded above by $1+\log _2 |\alpha \cap \beta|$. As this bound is twice as good as before, the distance between $H_i \cap F$ and $H_{i+1}\cap F$ must satisfy the same inequality. 
\end{proof}

\section{Complicated amalgamations}
\label{s:DistanceBoundSection}

The results of this section are due to T. Li when the index of $H$ is zero or one \cite{li:08}. The arguments presented here for the more general statements borrow greatly from these ideas.

\begin{lem}
\label{l:EulerBound}
\cite{li:08} Let $M$ be a compact, orientable, irreducible 3-manifold with incompressible boundary. Suppose $H$ and $Q$ are properly embedded surfaces in $M$ that are both incident to some component $F$ of $\bdy M$, where $Q$ is incompressible and $\bdy$-incompressible, and every loop and arc of $H \cap Q$ is essential on both surfaces. Then one of the following holds:
	\begin{enumerate}
		\item There is an incompressible, $\bdy$-incompressible surface $Q'$ which meets $H$ in fewer arcs than $Q$ did. The surface $Q'$ is either isotopic to $Q$, or is an annulus incident to $F$. 
		\item The number of arcs in $H \cap Q$ which are incident to $F$ is at most $(1-\chi(H))(1-\chi(Q))$. 
	\end{enumerate}
\end{lem}

\begin{proof}
There can be at most $1-\chi(H)$ non-parallel essential arcs on $H$, and at most $1-\chi(Q)$ non-parallel essential arcs on $Q$. Hence, if the number of arcs in $H \cap Q$ incident to $F$ is larger than $(1-\chi(H))(1-\chi(Q))$, then there must be at least two arcs $\alpha$ and $\beta$ of $H \cap Q$, incident to $F$, that are parallel on both $H$ and $Q$.  Suppose this is the case, and let $R_H$ and $R_Q$ denote the rectangles cobounded by $\alpha$ and $\beta$ on $H$ and $Q$, respectively. Note that $\alpha$ and $\beta$ can be chosen so that $A=R_H \cup R_Q$ is an embedded annulus. 

Since $\alpha$ is essential on $Q$ and $Q$ is $\bdy$-incompressible, it follows that $\alpha$ is essential in $M$. As $\alpha$ is also contained in $A$, we conclude $A$ is $\bdy$-incompressible. If $A$ is also incompressible then the result follows, as $A$ meets $H$ in fewer arcs than $Q$, and $A$ is incident to $F$. 

If $A$ is compressible then it must bound a 1-handle $V$ in $M$, since it contains an arc that is essential in $M$. The 1-handle $V$ can be used to guide an isotopy of $Q$ that takes $R_Q$ to $R_H$. This isotopy may remove other components of $Q \cap V$ as well. The resulting surface $Q'$ meets $H$ in at least two fewer arcs that are incident to $F$. 
\end{proof}

\begin{lem}
\label{l:SecondDistanceBound}
Let $M$ be a compact, orientable, irreducible 3-manifold with incompressible boundary, which is not an $I$-bundle. Suppose $H$ and $Q$ are properly embedded surfaces in $M$ that are both incident to some component $F$ of $\bdy M$, where $H$ is topologically minimal with respect to $\bdy M$, and $Q$ is an incompressible, $\bdy$-incompressible surface with maximal Euler characteristic. Then the distance between $H \cap F$ and $Q \cap F$ is at most $4+2(1-\chi(H))(1-\chi(Q))$.
\end{lem}

\begin{proof}
By Lemma \ref{l:EssentialIntersection} $H$ and $Q$ can be isotoped so that they meet in a collection of loops and arcs that are essential on both surfaces. Assume first $Q$ is not an annulus, and it meets $H$ in the least possible number of essential arcs. Then by Lemma \ref{l:EulerBound}, the number of arcs in $H \cap Q$ incident to $F$ is at most $(1-\chi(H))(1-\chi(Q))$. As each such arc has at most two endpoints on $F$, \[|H \cap Q \cap F| \le 2(1-\chi(H))(1-\chi(Q)).\] 

As in the proof of Corollary \ref{c:FirstDistanceBound}, when $F$ is not a torus we measure the distance between curves $\alpha$ and $\beta$ in $F$ in its curve complex. By \cite{hempel:01}, this distance is bounded above by $2+2\log _2 |\alpha \cap \beta|$. So we have, \[d(H \cap F, Q \cap F) \le 2+2\log _2 2(1-\chi(H))(1-\chi(Q)).\] But for any positive integer $n$, $\log _2 2n \le n$. Hence,
\begin{eqnarray*}
d(H \cap F, Q \cap F) &\le& 2+2(1-\chi(H))(1-\chi(Q))\\
&\le & 4+2(1-\chi(H))(1-\chi(Q)).
\end{eqnarray*}

When $F \cong T^2$, the distance between curves $\alpha$ and $\beta$ is measured in the Farey graph. In this case their distance is bounded above by $1+\log _2 |\alpha \cap \beta|$. As this bound is twice as good as before, the distance between $H \cap F$ and $Q \cap F$ must satisfy the same inequality. 

If $Q$ is an annulus, then since $M$ is not an $I$-bundle we may apply  Lemma 3.2 of \cite{li:08}, which says that $Q \cap F$ is at most distance 2 away from $Q' \cap F$, for some incompressible, $\bdy$-incompressible annulus $Q'$ that meets $H$ in the least possible number of essential arcs. By Lemma \ref{l:EulerBound}, the number of arcs in $H \cap Q'$ incident to $F$ is at most $(1-\chi(H))(1-\chi(Q'))$. As above, this implies the distance between $H \cap F$ and $Q' \cap F$ is at most \[2+2(1-\chi(H))(1-\chi(Q')).\] It follows that the distance between $H \cap F$ and $Q \cap F$ is at most $4+2(1-\chi(H))(1-\chi(Q))$. 
\end{proof}

\begin{lem}
\label{l:ThirdDistanceBound}
Let $M$ be a compact, orientable, irreducible 3-manifold with incompressible boundary, which is not an $I$-bundle. Suppose $H$ and $Q$ are properly embedded surfaces in $M$ that are both incident to some component $F$ of $\bdy M$, where $H$ is topologically minimal and $Q$ is an incompressible, $\bdy$-incompressible surface with maximal Euler characteristic. Then either $H$ is isotopic into a neighborhood of $\bdy M$, or the distance between $H \cap F$ and $Q \cap F$ is at most $4+3(1-\chi(H))(1-\chi(Q))$.
\end{lem}

\begin{proof}
If $H$ is not isotopic into a neighborhood of $\bdy M$ then by Corollary \ref{c:FirstDistanceBound} we may obtain a surface $H'$ from $H$ by a sequence of $\bdy$-compressions, which is topologically minimal with respect to $\bdy M$, where \[d(H \cap F,H' \cap F) \le 3\chi(H')-3\chi(H).\]

By Lemma \ref{l:SecondDistanceBound}, 
\[d(H' \cap F, Q \cap F) \le 4+2(1-\chi(H'))(1-\chi(Q)).\]

Putting these together gives:
\begin{eqnarray*}
d(H \cap F,Q \cap F) & \le & d(\bdy H, \bdy H') + d(\bdy H', \bdy Q)\\
& \le & 3\chi(H')-3\chi(H) + 4+2(1-\chi(H'))(1-\chi(Q))\\
&=&4+2(1-\chi(H))(1-\chi(Q))\\
&&+ (\chi(H')-\chi(H))(1+2\chi(Q))
\end{eqnarray*}

When $\chi(Q)<0$, then the expression $(\chi(H')-\chi(H))(1+2\chi(Q))$ is negative. Hence, we have

\begin{eqnarray*}
d(H \cap F,Q \cap F) & \le & 4+2(1-\chi(H))(1-\chi(Q))\\
&&+ (\chi(H')-\chi(H))(1+2\chi(Q))\\
& \le & 4+2(1-\chi(H))(1-\chi(Q))\\
& \le & 4+3(1-\chi(H))(1-\chi(Q))
\end{eqnarray*}

On the other hand, when $\chi(Q)=0$ then we have

\begin{eqnarray*}
d(H \cap F,Q \cap F) & \le & 4+2(1-\chi(H))(1-\chi(Q))\\
&&+ (\chi(H')-\chi(H))(1+2\chi(Q))\\
&=&6-3\chi(H)+\chi(H')\\
& \le & 6-3\chi(H)\\
&=&3+3(1-\chi(H))\\
&=&3+3(1-\chi(H))(1-\chi(Q))\\
&\le&4+3(1-\chi(H))(1-\chi(Q))
\end{eqnarray*}

\end{proof}

\begin{thm}
\label{t:DistanceBoundTheorem}
Let $X$  be a compact, orientable (not necessarily connected), irreducible 3-manifold with incompressible boundary, such that no component of $X$ is an $I$-bundle. Suppose some components $F_-$ and $F_+$ of $\bdy X$ are homeomorphic. Let $Q$ denote an incompressible, $\bdy$-incompressible (not necessarily connected) surface in $X$  of maximal Euler characteristic that is incident to both $F_-$ and $F_+$. Let $K=14(1-\chi(Q))$. Suppose $\phi:F_- \to F_+$ is a gluing map such that \[d(\phi(Q \cap F_-), Q \cap F_+) \ge Kg.\] Let $M$ denote the manifold obtained from $X$ by gluing $F_-$ to $F_+$ via the map $\phi$. Let $F$ denote the image of $F_-$ in $M$. Then any closed, topologically minimal surface in $M$ whose genus is at most $g$ can be isotoped to be disjoint from $F$.
\end{thm}

\begin{proof}
Suppose $H$ is a topologically minimal surface in $M$ whose genus is at most $g$. By Theorem \ref{t:OriginalIntersection} $H$ may be isotoped so that it meets $F$ in a collection of saddles, and so that the components of $H \setminus N(F)$ are topologically minimal in $M \setminus N(F)$. Note that $M \setminus N(F) = X'$, where $X' \cong X$. We denote the images of $F_-$ and $F_+$ in $X'$ by the same names. Let $H'=H \cap X'$. By Lemma \ref{l:EssentialBoundary}, $\bdy H'$ consists of essential loops on $\bdy X'$. When projected to $F$, these loops are all on the boundary of a neighborhood of the 4-valent graph $H \cap F$, it follows that the distance between $H' \cap F_-$ and $H' \cap F_+$ is at most one.

If $H$ could not have been isotoped to be disjoint from $F$, then the surface $H'$ can not be isotopic into a neighborhood of $\bdy X'$. We may thus apply Lemma \ref{l:ThirdDistanceBound} to obtain:

\begin{eqnarray*}
d(Q \cap F_-, Q \cap F_+) & \le & d(Q \cap F_-, H' \cap F_-)+ d(H' \cap F_-, H' \cap F_+)\\
&&+d(Q \cap F_+, H' \cap F_+)\\
&\le& 2(4+3(1-\chi(H'))(1-\chi(Q)))+1\\
& = & 9+6(1-\chi(H'))(1-\chi(Q))\\
& \le & 9+6(1-\chi(H))(1-\chi(Q))\\
&\le & 9+6(2g-1)(1-\chi(Q))\\
&=& 9-6(1-\chi(Q))+12g(1-\chi(Q))
\end{eqnarray*}

Note that the theorem only has content when $g \ge 2$, since by definition a torus cannot be topologically minimal. Also, $Q$ has non-empty boundary, so $1-\chi(Q) \ge 1$. It follows that  

\begin{eqnarray*}
d(Q \cap F_-, Q \cap F_+) &\le& 9-6(1-\chi(Q))+12g(1-\chi(Q)) \\
& < & 2g(1-\chi(Q)) + 12g(1-\chi(Q))\\
&=&14g(1-\chi(Q))\\
&=&Kg
\end{eqnarray*}
\end{proof}

\bibliographystyle{alpha}

\begin{thebibliography}{Hem01}

\bibitem[Baca]{StabilizationResults}
D.~Bachman.
\newblock Heegaard splittings of sufficiently complicated 3-manifolds {I}:
  {S}tabilization.
\newblock Preprint.

\bibitem[Bacb]{AmalgamationResults}
D.~Bachman.
\newblock Heegaard splittings of sufficiently complicated 3-manifolds {II}:
  {A}malgamation.
\newblock Preprint.

\bibitem[Bacc]{TopIndexI}
D.~Bachman.
\newblock Topological {I}ndex {T}heory for surfaces in 3-manifolds.
\newblock Preprint.

\bibitem[Bac02]{crit}
D.~Bachman.
\newblock Critical {H}eegaard surfaces.
\newblock {\em Trans. Amer. Math. Soc.}, 354(10):4015--4042 (electronic), 2002.

\bibitem[BJ]{existence}
D.~Bachman and J.~Johnson.
\newblock On the existence and non-existence of topologically minimal surfaces.
\newblock In preparation.

\bibitem[CG87]{cg:87}
A.~J. Casson and C.~McA. Gordon.
\newblock Reducing {H}eegaard splittings.
\newblock {\em Topology and its Applications}, 27:275--283, 1987.

\bibitem[Hem01]{hempel:01}
J.~Hempel.
\newblock 3-manifolds as viewed from the curve complex.
\newblock {\em Topology}, 40:631--657, 2001.

\bibitem[Li]{li:08}
T.~Li.
\newblock Heegaard surfaces and the distance of amalgamation.
\newblock Preprint.

\end{thebibliography}

\end{document}